\documentclass{amsart}
\usepackage{amsmath,amssymb}
\usepackage{hyperref}
\usepackage{xcolor}
\usepackage{soul}
\newcommand{\Set}[1]{\left\{\, #1 \,\right\}}

\newcommand{\Span}[1]{\langle\, #1 \,\rangle}

\newcommand{\Order}[1]{\lvert #1 \rvert}
\newcommand{\Index}[2]{\lvert #1 : #2 \rvert}
\providecommand{\Size}[1]{\left| #1 \right|}

\DeclareMathOperator{\Frat}{\Phi}
\DeclareMathOperator{\Hom}{Hom}
\DeclareMathOperator{\Irr}{Irr}

\DeclareMathOperator{\dt}{bdt}
\DeclareMathOperator{\bdt}{bdt}
\DeclareMathOperator{\rexp}{\mathit{e}}
\DeclareMathOperator{\br}{\mathit{b}}
\DeclareMathOperator{\cl}{cl}
\DeclareMathOperator{\cn}{\mathit{k}}

\renewcommand{\phi}[0]{\varphi}
\renewcommand{\theta}[0]{\vartheta}
\renewcommand{\epsilon}[0]{\varepsilon}

\newcommand{\N}{\text{$\mathbb{N}$}}
\newcommand{\Z}{\text{$\mathbb{Z}$}}

\newcommand{\C}{\text{$\mathbb{C}$}}

\newcommand{\F}{\text{$\mathbb{F}$}}

\theoremstyle{plain}

\newtheorem{dummy}{Dummy}
\numberwithin{dummy}{section}
\numberwithin{equation}{section}

\newtheorem{theorem}[dummy]{Theorem}
\newtheorem{lemma}[dummy]{Lemma}

\newtheorem{proposition}[dummy]{Proposition}
\newtheorem{corollary}[dummy]{Corollary}

\theoremstyle{definition}

\theoremstyle{remark}

\newtheorem{remark}[dummy]{Remark}

\begin{document}

\date{August 9, 2021}

\title%[The breadth-degree type of a finite $p$-group]
{The breadth-degree type of a finite $p$-group}

% \title%[Breadth and character degree type of a finite $p$-group]
% {Breadth and character degree type of a finite $p$-group}

% \author{N. Gavioli, A. Mann and V. Monti}

\author[N. Gavioli]{Norberto Gavioli} \address[Norberto Gavioli]{
  Dipartimento di Ingegneria e Scienze dell'Informazione e Ma\-te\-ma\-ti\-ca\\
  Universit\`a degli Studi dell'Aquila\\
  Via Vetoio\\
  I-67100 Coppito (L'Aquila)\\
  Italy} \email{norberto.gavioli@univaq.it}

\author[A. Mann]{Avinoam Mann} \address[Avinoam Mann]{
  Einstein Institute of Mathematics\\
  The Hebrew University of Jerusalem\\
  Manchester House 205\\
  Givat       Ram.       Jerusalem,      9190401,       Israel       }
\email{avinoam.mann@mail.huji.ac.il} \author[V.  Monti]{Valerio Monti}
\address[Valerio Monti]{
  Dipartimento di Scienza e Alta Tecnologia\\
  Universit\`a degli Studi dell'Insubria\\
  Via Valleggio, 11 \\ I-22100 Como\\
  Italy} \email{valerio.monti@uninsubria.it}

\thanks{The authors are  indebted to Carlo Maria  Scoppola for helpful
  discussions and  suggestions in the  preparation of this  work.  The
  first  and  third  authors  are members  of  GNSAGA--INdAM  and  are
  grateful  to  \emph{Istituto  Nazionale   di  Alta  Matematica  ``F.
    Severi"} (Roma) for support and hospitality during the preparation
  of this paper.  }

% \affiliationone{% in this example, two authors share an institution
%	Norberto Gavioli\\
%	Dipartimento di Ingegneria e Scienze dell'Informazione e Ma\-te\-ma\-ti\-ca\\
%	Universit\`a degli Studi dell'Aquila\\
%	Via Vetoio\\
%	I-67100 Coppito (L'Aquila)\\
%	Italy \email{norberto.gavioli@univaq.it}}
%% Important: Do not put any empty line here.
% \affiliationtwo{% in this example, two authors share an institution
%	Avinoam Mann\\
%	Einstein Institute of Mathematics\\
%	The Hebrew University of Jerusalem\\
%	Manchester House 205\\
%	Givat       Ram.        Jerusalem,       9190401,       Israel
% \email{avinoam.mann@mail.huji.ac.il}}
% \affiliationthree{Valerio Monti\\
%	Dipartimento di Scienza e Alta Tecnologia\\
%	Universit\`a degli Studi dell'Insubria\\
%	Via Valleggio, 11 \\ I-22100 Como\\
%	Italy \email{valerio.monti@uninsubria.it} }

% \address[Norberto Gavioli]%
% {Dipartimento di Ingegneria e Scienze dell'Informazione e Ma\-te\-ma\-ti\-ca\\
% Universit\`a degli Studi dell'Aquila\\
% Via Vetoio\\
% I-67100 Coppito (L'Aquila)\\
% Italy} \email{norberto.gavioli@univaq.it}
%
% \author{Valerio Monti}
%
% \address[Valerio Monti]
% {Dipartimento di Scienza e Alta Tecnologia\\
% Universit\`a degli Studi dell'Insubria\\
% Via Valleggio, 11 \\ I-22100 Como\\
% Italy}
%
% \email{valerio.monti@uninsubria.it}

%
%
%



% \classno{20D15 (primary), 20C15 (secondary)}

\begin{abstract}
  % We determine the upper bound $|G|\le p^{\frac{b(3b+4d-1)}{2}}$ for
  % the size  of a stem finite  $p$-group $G$ in terms  of its breadth
  % $b$ and its maximum character degree $p^d$.
  In the  present paper we show  that a stem finite  $p$-group $G$ has
  size                            bounded                           by
  $\min\left(p^{(8d        -        2\log_2d       +        b        -
      4)(b+1)/2},p^{b(3b+4d-1)/2}\right)$ where $b$  is the breadth of
  $G$  and  $p^d$  is  the  maximum character  degree  of  $G$.  As  a
  consequence  there are  only  finitely many  finite stem  $p$-groups
  having breadth $b$ and maximum character degree $p^d$.
\end{abstract}

\maketitle

\thispagestyle{empty}

\section*{Introduction}

The problem  of finding all  finite groups of  a given order  is quite
old. In  the case of finite  $p$-groups (where $p$ is  a prime number)
the task  is particularly  hard as the  number of  isomorphism classes
grows rapidly with  the order.  This particular  difficulty led Philip
Hall, aiming  at achieving a \emph{systematic  classification theory},
as he himself writes in his famous paper \cite{Hall1940}, to introduce
the  concept  of  isoclinism,  an equivalence  relation  among  groups
preserving the commutation structure. A group $T$ such that its center
$Z(T)$  is contained  in the  derived subgroup  $T'$ is  said to  be a
\emph{stem  group}. Philip  Hall showed  that finite  stem groups  are
exactly the ones of minimum order  within an isoclinism class. He also
pointed out  that the ratio  of the number  of conjugacy classes  of a
given size of a  finite group $G$ with respect to the  order of $G$ is
preserved by  isoclinism.  In \cite{Hall1940}  it is also  stated that
the ratio  of the number of  irreducible characters of $G$  of a given
degree with respect to the order of $G$ is invariant under isoclinism.
A proof of this result is given in \cite[Theorem D]{Mann1999}.  Adding
up over all the possible lengths of conjugacy classes (or over all the
possible degrees of  the irreducible characters) we see  that also the
\emph{commutativity degree} $\cn(G)/\Size G$ of  a finite group $G$ is
invariant  under  isoclinism.    This  number  is  known   to  be  the
probability   that  two   randomly   chosen   elements  commute   (see
\cite{Erfanian2007,Guralnick2006,Gustafson1973,Lescot1995,Rusin1979}).

Another invariant under isoclinism  which has been extensively studied
is the  breadth $\br(G)$ of a  finite $p$-group $G$. It  is defined as
the maximum  integer $b$ such that  there exists a conjugacy  class in
$G$ of size $p^b$.  A finite $p$-group  has breadth $1$ if and only if
its derived subgroup has order  $p$. This result is usually attributed
to H-G.~Knoche  \cite{Knoche1951}, although Rolf Brandl  recently made
the authors aware that it  was already known to W.~Burnside \cite[Chap
VIII  \S  99  p.~126,  Ex.~1]{Burnside1955}.  The  breadth  $b$  of  a
$p$-group  $G$  gives   information  on  the  structure   of  $G$:  in
\cite{Vaughan-Lee1974}  M.R.~Vaughan-Lee   showed  that   the  derived
subgroup of $G$ has order bounded by $p^{\frac{b(b+1)}{2}}$.  A famous
conjecture attributed  to Peter~M.~Neumann stated that  the nilpotency
class  $\cl(G)$  of $G$  is  at  most  $\br(G)+1$.  Known  bounds  are
$\cl(G)    \le    \frac{p}{p-1}\br(G)+1$    by    C.~R.~Leedham-Green,
P.~M.~Neumann    and     J.~Wiegold    \cite{Leedham-Green1969}    and
$\cl(G)\le    \frac{5}{3}\br(G)+1$   for    $p=2$   by    M.Cartwright
\cite{Cartwright1987}. Some counterexamples to the original conjecture
have been  provided by  V.~Felsch \cite{Felsch1980} and  by W.~Felsch,
J.~Neubüser  and  W.~Plesken  \cite{Felsch1981}  and  B.~Eick  et  al
\cite{Eick2006}.  A.~Mann proved in \cite{Mann2006} that the subgroup of
a finite $p$-group  generated by the elements  having minimal positive
breadth  has  nilpotency class  at  most  $3$.  Finite  $p$-groups  of
breadth     $2$      and     $3$     have     been      studied     in
\cite{Gavioli1998,Parmeggiani1999}.

The breadth alone,  however, is far from being  a sufficient parameter
in  the  attempt  to  enumerate   the  isoclinism  classes  of  finite
$p$-groups: indeed  in Section~\ref{sec:unbounded_size}, we  see that,
given an integer $b\ge 1$, it is not difficult to construct infinitely
many isoclinism classes  of finite $p$-groups having  the same breadth
$b$.

In  a somewhat  dual fashion  one can  consider the  set $\Irr(G)$  of
irreducible characters  of $G$ and the  \emph{representation exponent}
$\rexp(G)$          defined          by          the          equality
$p^{\rexp(G)}=\max\Set{\chi(1) \mid \chi\in \Irr(G)}$. This is also an
isoclinism invariant.   Trivially $\rexp(G)=0$ if  and only if  $G$ is
abelian. The $p$-groups $G$  with $\rexp(G)=1$ have been characterized
in  \cite{Isaacs1964a}  where  it  is   shown  that  a  necessary  and
sufficient condition for $\rexp(G)$ to be  equal to $1$ is that either
$G$ has a maximal subgroup which  is abelian or that the center $Z(G)$
has index in $G$  at most $p^3$ (both conditions can  hold at the same
time, though). Incidentally  this last condition implies  that $G$ has
breadth  $2$,  whereas,  on  the contrary,  there  are  $p$-groups  of
arbitrarily   large  breadth   with  a   maximal  subgroup   which  is
abelian. The $p$-groups with representation  exponent at most $2$ have
been studied in \cite{Passman1966a}.

In  Section~\ref{sec:unbounded_size} we  show that  the representation
degree  alone  is not  suitable  for  an  enumeration attempt  of  the
isoclinism  classes of  finite $p$-groups  (see \eqref{eq:notsuitable}
below).

S.R.~Blackburn  showed  in  \cite[Theorem 1]{Blackburn1994}  that  the
number  of $p$-groups  of given  order within  an isoclinism  class is
relatively small with  respect to the number of  all finite $p$-groups
of that order, so  that we expect to find ``not  too many" stem groups
within an isoclinism class.

\bigskip

In this  paper we  define the \emph{breadth-degree  type} of  a finite
$p$-group $G$  to be  the pair  $(\br(G),\rexp(G))$.  We  approach the
problem of  giving a bound  for a  stem group of  given breadth-degree
type in two different ways giving two quadratic bounds respectively in
Corollary~\ref{cor:vl} and in Theorem~\ref{thm:main}.

In  Theorem~\ref{thm:main} we  prove that  a stem  $p$-group of  given
breadth-degree     type    $(b,d)$     has     order    bounded     by
$p^{\frac{b(3b+4d-1)}{2}}$ .  It follows  that there are only finitely
many stem $p$-groups of given breadth-degree type.  This suggests that
the breadth-degree type is a parameter to be taken into account in any
enumeration  procedure  (up  to   isoclinism)  of  classes  of  finite
$p$-groups.

% Isoclinism casses  of finite $p$-group of  small breadth-degree type
% are studied in greater detail  in paper in preparation (A. Cupaiolo,
% N.   Gavioli, A.   Morresi Zuccari,  and C.   M.  Scoppola.   Finite
% $p$-groups of small bradth-degree type.  2019).

\section{Notation and preliminaries}
The groups  considered in this  paper will be finite  $p$-groups where
$p$  is  a prime.   We  will  use  standard notation.   In  particular
$\gamma_i(G)$ and $Z_i(G)$ will denote respectively the $i$-th term of
the lower and the upper central series of $G$.  If $x$ and $y$ are two
elements  of a  group  then the  commutator  $x^{-1}y^{-1}xy$ will  be
denoted  by $[x,y]$  and when  $H$  is a  subset of  $G$ the  notation
$[H,x]$ will be used to denote the set $\Set{[h,x] \mid h\in H}$. When
$H\le G$  is a subgroup  and $x\in Z_2(G)$ we  have that $[H,x]$  is a
subgroup of the  center of $G$.  The set of  irreducible characters of
$G$ will be denoted by $\Irr(G)$. The breadth $\br_G(g)$ of an element
$g\in    G$   is    defined    by   $\Size{g^G}=p^{\br_G(g)}$    where
$g^G=\Set{x^{-1}gx \mid x\in G}$ is the  conjugacy class of $g$ in $G$
and $\br(G)=\max \Set{\br_G(g) \mid g\in G}$ is said to be the breadth
of the  group $G$. The  subscript is often  omitted where there  is no
ambiguity and we will write $\br(g)$ to mean $\br_G(g)$.  We denote by
$B_i(G)$ the subgroup  of $G$ generated by the  elements whose breadth
is  at most  $i$.  The  representation exponent  $\rexp(G)$ of  $G$ is
defined by $p^{\rexp(G)}=\max\Set{\chi(1)  \mid \chi\in \Irr(G)}$. The
number of  the conjugacy classes  of $G$  is denoted by  $\cn(G)$.  We
shall  say  that  a  finite  $p$-group  $G$  has  breadth-degree  type
$\bdt(G)=(b,d)$ if $b=\br(G)$
% is the breadth of $G$
and $d=\rexp(G)$.
% is the representation exponent of $G$.
% In this case we shall write $\dt(G)=(b,d)$.

We remind the reader that two finite groups $G$ and $H$ are said to be
isoclinic if  there are  isomorphims $\alpha\colon G/Z(G)  \to H/Z(H)$
and $\beta\colon G'  \to H'$ such that  $\beta([g,g'])=[h,h']$ for all
$(g,g')\in  G\times  G$ and  $(h,h')  \in  H\times H$  satisfying  the
conditions                  $hZ(H)=\alpha(gZ(G))$                  and
$h'Z(H)=\alpha(g'Z(G))$. Isoclinism  is an equivalence  relation among
finite groups.  A group $T$ is said to be stem if its center $Z(T)$ is
contained in the derived subgroup $T'$ and stem groups are exactly the
groups having minimum order in  some isoclinism class.  In the present
paper we shall need  to compare the order of a  $p$-group $G$ with the
order of  stem group  in the  isoclinism class  of $G$:  the following
remark will be useful to this purpose.

\begin{remark} \label{rem:isoclinicgroupsizes}  Let \(G\) be  a finite
  group.                                                         Since
  \(\Size{Z(G)\cap         G'}=\Size{G'}/\Size{G'Z(G)/Z(G)}=\Size{G'}/
  \Size{\left(G/Z(G)\right)'}\)  we see  that $\Size{Z(G)\cap  G'}$ is
  invariant            under           isoclinism.             Clearly
  \(\Size{G}\ge   \Index{G}{Z(G)}\cdot\Size{Z(G)\cap  G'}\)   and  the
  second member is invariant by isoclinism. Equality holds if and only
  if \(Z(G)\le G'\), that is, \(G\) is a stem group. Thus, stem groups
  are  precisely those  of minimum  order in  their isoclinism  class,
  hence  $ \Index{G}{Z(G)}\cdot\Size{Z(G)\cap  G'}$ is  the size  of a
  stem group $T$ in the isoclinism class of $G$.  In particular
  \begin{equation*}%\label{eq:sizeofisoclinicG}
    \begin{split}
      \Size G &= \Index{G}{Z(G)} \cdot \Index{Z(G)}{(Z(G)\cap G')}\cdot \Size{Z(G)\cap G'}  \\
      &= \Size T\cdot \Index{Z(G)}{(Z(G)\cap G')}.
    \end{split}
  \end{equation*}
  See also \cite[Corollary 29.4]{Berkovich2008}.
\end{remark}

% \begin{definition}
%   \label{def:degree_type}
%   We shall say  that a finite $p$-group $G$  has breadth-degree type
%   $(b,d)$ if $b=\br(G)$ is the breadth of $G$ and if $d=\rexp(G)$ is
%   the representation exponent  of $G$.  In this case  we shall write
%   $\dt(G)=(b,d)$.
% \end{definition}

It  was  already  known  to  Philip Hall  that  the  breadth  and  the
representation  degree   of  a   finite  group  are   invariant  under
isoclinism:  two isoclinic  $p$-groups  have  the same  breadth-degree
type.

We define the size $\sigma(b,d)$ of the breadth-degree type $(b,d)$ as
\begin{equation*}
  % \sigma(b,d)=\sup\Set{|G| \mid G \text{ stem $p$-group of
  % breadth-degree type $(b',d')$ with $b'\le b$ and $d'\le d$}}.
  \sigma(b,d)=\sup\Set{|G| \mid G \text{ stem $p$-group, $\dt(G)=(b',d')$ with $b'\le b$ and $d'\le d$}}.
\end{equation*}

We will  use, without  further mention,  the fact that  if \(G\)  is a
finite \(p\)-group  with $\dt(G)=(b,d)$ and  \(H\) is a quotient  or a
subgroup of \(G\) with $\dt(H)=(b',d')$, then $b'\le b$ and $d'\le d$.

\section{Bounding  the size  of  stem groups  of fixed  breadth-degree
  type}

We start with a lemma.
\begin{lemma}\label{lem:indexofcentre}
  Let the  index of  a subgroup  $H$ of a  group $G$  be finite  and a
  product of $k$ primes (counting multiplicities). If $G$ is generated
  by  $H$ and  a subset  $S$,  then it  is  generated by  $H$ and  $k$
  elements of $S$.
\end{lemma}
\begin{proof}
  If $H \neq G$, we can find  $x\in S$ such that $x\notin H$, and then
  $ \langle H, x\rangle$ contains $H$ properly. Continuing in the same
  way, we construct a  series $H = H_0 < H_1 < \dots  < H_n = G$, such
  that for each $i$ we have $H_i = \langle H_{i-1}, x_i\rangle $, with
  $x_i \in  S$. Then $G  = \langle H, x_1,  \dots , x_n\rangle  $, and
  $n \le k$.
\end{proof}

\begin{theorem}\label{thm:indexofcentre}
  Let the finite $p$-group $G$ be  generated by elements of breadth at
  most   $b$,   and   have    representation   exponent   $d$.    Then
  $\Index{G}{Z(G)} \le p^{(4d - \log_2d - 2)(b+1)}.$
\end{theorem}
\begin{proof}
  By  \cite[Theorem A]{Isaacs1965},  $G$  contains  a maximal  abelian
  subgroup $H$ such  that $|G:H|$ divides $ p^{4d -  \log_2d - 2}$. By
  Lemma~\ref{lem:indexofcentre}, $G$  is generated by $H$  and at most
  $4d - \log_2d  - 2$ elements $x_1,  \dots , x_n$ of  breadth at most
  $b$.   Then $Z(G)  =  \cap_i  C_G(x_i) \cap  H$  has  index at  most
  $\Index{G}{H}\cdot \Pi_i \Index{G}{C_G(x_i)} \le  p^{(4d - \log_2d -
    2)(b+1)}$.
\end{proof}

\begin{corollary}\label{cor:wieg}
  Under  the  assumptions of  Theorem~\ref{thm:indexofcentre},  assume
  also that  $G$ is a stem  group.  Then $|G| \le  p^{n(n+1)/2}$ where
  $n=(4d - \log_2 d - 2)(b+1)$.
\end{corollary}
\begin{proof}
  By   \cite[Theorem~2.1]{Wiegold1965}  a   $p$-group  $X$   in  which
  $|X:Z(X)| = p^z $ satisfies  the inequality $|X'| \le p^{z(z-1)/2}$.
  In   our   case    $z   \le   n$   and   $Z(G)    \le   G'$,   hence
  $|G| \le \Index{G}{Z(G)}\Size{G'} \le p^{n(n+1)/2}$.
\end{proof}

\begin{corollary}\label{cor:vl}
  If  $G$ is  a stem  $p$-group of  breadth-degree type  $(b,d)$, then
  $\Size{G} \le  p^{(8d - 2\log_2d +  b - 4)(b+1)/2}$. In  other words
  $\sigma(b,d)\le p^{(8d - 2\log_2d + b - 4)(b+1)/2}$.
\end{corollary}
\begin{proof}
  The proof  mimics the  one of  the previous  corollary by  using the
  bound $\Size{G'} \le p^{b(b+1)/2}$ in \cite{Vaughan-Lee1974}.
\end{proof}
Note  that Corollary~\ref{cor:vl}  requires that  $G$ has  breadth $b$
whereas Corollary~\ref{cor:wieg}  has the  weaker hypothesis  that $G$
can be generated by elements of breadth at most $b$.

% \proclaim {Corollary 2} If $G$ is a stem $p$-group of breadth-degree
% type                          $(b,d)$,                          then
% $|G| \le p^{(8d - 2\log_2d + b - 4)(b+1)/2}$.\endproclaim
%
% To   see  this,   replace   the  previous   bound   for  $|G'|$   by
% M.R.Vaughan-Lee's bound $|G'| \le p^{b(b+1)/2}$ [VL].

\section{A different bound}

In this  section we provide  a bound for $\sigma(b,d)$  different from
the one  in Corollary~\ref{cor:vl}.  This  is done by finding  a group
$K/L$, where $L \triangleleft K \le G$, such that $b(K/L) < b(G)$, and
applying induction.

\begin{lemma}\label{lem:genbound_first_reduction}
  Let    $G$    be    a     non-abelian    stem   \(p\)-group    and    let
  \(b^{*}=\min_{g\in  G\setminus Z(G)}\br_G(g).\)  The group $G$  has a
  stem quotient of  order $\frac{\Order{G}}{p^{b^{*}-1}}$ whose second
  centre contains an element of breadth 1.
\end{lemma}

\begin{proof}
  By  \cite   [Theorem  2]{longobardi1999},  we  can   choose  $g$  in
  $Z_2(G)\setminus  Z(G)$  of  breadth  $b^{*}$.  Thus  $[g,G]$  is  a
  subgroup of $Z(G)$ of size $p^{b^{*}}$ and if we take a subgroup $N$
  of  $[g,G]$  of  size  $p^{b^{*}-1}$, the  quotient  $G/N$  has  the
  requested  size and  $gN$  is an  element of  the  second center  of
  breadth 1. To prove that $G/N$ is  stem, note that if $xN$ is in the
  centre of $G/N$, then $[x,G]\le  N$ so that $\br_{G}(x)<b^*$. By the
  minimality of $b^*$, we find that $x  \in Z(G)$. Thus $x \in G'$ and
  $xN \in (G/N)'$.
\end{proof}

\begin{lemma}\label{2nd_lemma_abelian_breadth_1}
  Let  $G$ be  a finite  $p$-group of  breadth $b$  and let  $N$ be  a
  minimal  normal  subgroup.   Assume  that there  exists  an  element
  $g\in G$  such that $[g,G]=N$. If  $M=C_G(g)$ then $M$ is  a maximal
  subgroup  of   $G$  containing   $N$  and   $\br(M/N)<b$.   Moreover
  $\br_{G/N}(yN)<\br_G(y)\le  b$ for  every $y  \in G\setminus  M$ and
  $[h,G]\supseteq N$ if $\br_G(h)=b$ and $C_G(h)\nleq M$.
%
  % Let $G$ be a finite $p$-group, $N$ a minimal normal subgroup,
  % $g\in C_G(G/N)\setminus Z(G)$ and $M=C_G(g)$. The subgroup $M$ is
  % a maximal subgroup of $G$ and if $x \in G$ has maximal breadth
  % $\br(x)=\br(G)$ then either $x^G \supseteq xN$ or
  % $x^G \cap xN=\Set x$, in which second case $x\in C_G(x)\le M$ and
  % $\br_M(x)=\br_G(x)-1$.  As a consequence $\br(M/N) \le \br(G)-1$
  % where equality holds if $\br(M)=\br(G)$.
%
\end{lemma}
\begin{proof}
  Clearly $N$  is central and $\br(g)=1$,  so that $M$ is  maximal and
  contains $N$. Since $N$ is central,  $[y,G] \cap N$ is a subgroup of
  $N$ for  every $y \in  G$. Provided that  $[y,G] \cap N\ne  1$, this
  implies     that    $[y,G]\supseteq     N$    and,     consequently,
  $\br_{G/N}(yN)<\br_G(y)\le   b$.   Note   that  this   happens  when
  $y \in G\setminus M$ for $[y,g]$ is a non-trivial element of $N$.

  Let now  $x$ be an  element of $M$ and  let $C=C_G(x)$. If  $C\le M$
  then $\br_{M/N}(xN)\le \br_M(x)<\br_G(x)\le  b$. Otherwise $G=CM$ so
  that $x^G=x^M$. Thus $(xg)^G \supseteq (xg)^M=x^Mg=x^Gg$ and working
  modulo   $N$  we   get  $\br_{G/N}(xgN)\ge   \br_{G/N}(xN)$.   Given
  $c \in C\setminus  M$, we have that $[xg,c]=[g,c]$  is a non-trivial
  element      of      $[xg,G]\cap      N$      and      as      above
  $\br_{M/N}(xN)\le  \br_{G/N}(xN)\le \br_{G/N}(xgN)<\br_G(xg)\le  b$.
  This shows that \(\br(M/N)< b\).
  
  % Note that $[x,G]\supseteq N$ if
  Suppose that \(h\in  G\) and that $\br_G(h)=b$  and $C_G(h)\nleq M$.
  If $h\notin M$, we have already  noted that $[h,G] \supseteq N$.  If
  $h\in M$,  then $G = C_G(h)M$  implies that $[h,G] =  [h,M]$ so that
  $b_M(h)  = b$.  Since  $b_{M/N}(hN) \le  b(M/N) <  b$,  and this  is
  possible only if $[h,M] \supseteq N$.
%
%
%
  % Note that $M$ is a maximal subgroup of $G$.  Suppose that
  % $x^G \not\supseteq xN$, i.e.\ $x^G \cap xN=\Set x$, and assume by
  % way of contradiction that the centralizer $C=C_G(x)$ of $x$ in $G$
  % is not contained in $M$.  We have $G=CM$ which implies that
  % $x^G=x^M$. If $c\in C\setminus M$ then $1\ne [g,c]\in N$, so that
  % $(gx)^G \supseteq g(x^M) \cup \Set{gx[g,c]} \supsetneq
  % g(x^M)=g(x^G)$.
  % In particular $\Size{(gx)^G} > \Size{g(x^G)} = p^{\br(G)}$ which
  % is impossible. It follows that $C_G(x)\le M$ as claimed.
\end{proof}

Part of  the proof  of item~\ref{lem:reduction_to_maximal_itemtwo_new}
of     the     following     lemma    has     been     inspired     by
\cite[Proposition~4.1]{Isaacs1964}.

\begin{lemma}  \label{lem:reduction_to_maximal_new}   Let  $G$   be  a
  non-abelian   $p$-group  of   breadth-degree   type  $(b,d)$.    Let
  $M\lessdot  G$   be  a  maximal  subgroup   containing  $Z(G)$,  let
  $N\le  Z(G)$  be of  order  $\Size  N=p$  and set  $C=C_G(G/N)$  and
  $D=C_G(C)$.
  \begin{enumerate}
  \item \label{lem:reduction_to_maximal_itemone_new} If  $x$ is chosen
    of minimal  breadth $t=\br_G(x)$  in the  set $G\setminus  M$ then
    $\Index{Z(M)}{Z(G)}\le                   p^t$                  and
    $\Index{Z(M)M'}{M'}\le   p^{2t}    \cdot\Index{Z(G)G'}{G'}$.    In
    particular $\Index{Z(M)M'}{M'}\le p^{2t}$ when $G$ is stem.
  \item
    \label{lem:reduction_to_maximal_itemtwo_new} %If $G$ is stem then
    $\Index{C}{Z(G)}=\Index{G}{D}\le      \chi(1)^2$      for      all
    $\chi\in \Irr(G)$ such that $N\nleq \ker\chi$.
  \item \label{lem:reduction_to_maximal_itemthree_new}
    If %$G$ is stem and
    $\Index{G}{D}=  \chi(1)^2$ for  some $\chi\in  \Irr(G)$ such  that
    $N\nleq \ker\chi$  then $C\cap D=Z(G)$, $N\cap  D'=1$, $G'\le ND'$
    and  $G=CD$ is  a  central product.   If  $\chi(1)>1$ then  $C'\ne
    1$. Moreover if  $G$ is stem then $D/N$ is  stem, $D$ is isoclinic
    to $D/N$ and, provided that $\chi(1) > 1$, we have also $b(D)<b$.
  \item  \label{lem:reduction_to_maximal_itemfour_new} If  $C$ is  not
    abelian then there  exists $g\in Z_2(G)$ such  that $\br(g)=1$ and
    $\min \Set{b(x)\mid x\in G\setminus C_G(g)}= 1$.
  \item   \label{lem:reduction_to_maximal_itemfive_new}   If  $C$   is
    abelian              then             $b(G/N)=b-1$              or
    $\Index{G}{D}= \Index{C}{Z(G)}\le p^{\min\{b,2d-1\}}$.
  \end{enumerate}
  In      particular      $\Index{Z(M)M'}{M'}\le      p^{2b}$      and
  $\Index{C}{Z(G)}=\Index{G}{D}\le p^{2d}$ when $G$ is stem.
\end{lemma}
\begin{proof}
  The map  $\phi\colon m\mapsto [m,x]$  is an endomorphism  of $Z(M)$.
  Since $M$  is a maximal subgroup  in $G$ we have  that $G=\Span{x}M$
  and that the kernel of $\phi$  is $Z(G)$.  If we compute the breadth
  of                   $x$                   we                   find
  $p^t=  \br(x)=  \Size{[G,x]}  =  \Size{[M,x]}\ge  \Size{[Z(M),x]}  =
  \Size{\phi(Z(M))}  =   \Index{Z(M)}{\ker\phi}  =\Index{Z(M)}{Z(G)}$.
  Note         that          $G'=[x,M]M'$         implies         that
  $\Index{G'}{M'}\le \Size{[x,M]}= p^t$. We have
  % Since $G$ is a stem group and $Z(G)\le G'$, we have
  \begin{align*}
    \Index{Z(M)M'}{M'} &\le  \Index{Z(M)G'}{M'} = \Index{Z(M)G'}{Z(G)G'}\cdot\Index{Z(G)G'}{G'}\cdot  \Index{G'}{M'}   \\
                       &\le  \Index{Z(M)}{Z(G)} \cdot\Index{Z(G)G'}{G'} \cdot  \Index{G'}{M'} \le p^{2t} \cdot \Index{Z(G)G'}{G'}.
  \end{align*}
  If     $G$    is     stem    then     $Z(G)\le    G'$     so    that
  $\Index{Z(M)M'}{M'}\le         p^{2t}$.          This         proves
  item~\ref{lem:reduction_to_maximal_itemone_new}.  \medskip

  Since $[G',Z_2(G)]=1$  and since  $C\le Z_2(G)$  we have  $G'\le D$.
  Also  $[y^p,c]=[y,c]^p=1$ for  all $y\in  G$ and  $c\in C$,  so that
  $\Frat(G)=G^pG'\le D$  and $G/D$  is elementary  abelian. As  $N$ is
  cyclic         of         order         $p$,         the         map
  $cZ(G)\mapsto [\cdot,c]\in  \Hom(G/D,N)\cong G/D$ is easily  seen to
  be an injective homomorphism of  $C/Z(G)$ into $G/D$. Dually the map
  $gD\mapsto  [\cdot,g]\in  \Hom(C/Z(G),N)\cong   C/Z(G)$  defines  an
  injective homomorphism  of $G/D$ into  $C/Z(G)$. We have  shown that
  $\Index{C}{Z(G)}=\Index{G}{D}$.

  As   $\bigcap_{\chi  \in   \Irr(G)}   \ker   \chi=1$  there   exists
  $\chi\in  \Irr(G)$  such  that   $N\nleq  \ker\chi$,  in  particular
  $N\cap    \ker\chi=1$    since    $N$    has    order    $p$.     If
  $1\ne  z\in   N  \le   Z(G)$  then   $\chi(z)=\chi(1)\theta$,  where
  $1\ne  \theta\in \C$  is a  $p$-th root  of $1$.   Let $y\in  G$. If
  $y\notin D$ there  exists $c\in C$ such that $1\ne  [y,c]=z\in N$ so
  that   $\chi(y)=\chi(y^c)=\chi(y[y,c])=\chi(yz)=\theta\chi(y)$.   It
  follows   that   $0=(1-\theta)\chi(y)$,   which   in   turn   forces
  $\chi(y)=0$.
  % Reversing the roles of $c$ and $y$ we have also that $\chi(c)=0$
  % for $c\in C\setminus Z(G)$.
  We  have  shown that  $\chi(y)\ne  0$  implies  $y\in  D$ and  as  a
  consequence
  \begin{equation*}
    1=\frac{1}{\Size{G}}\sum_{y\in G}\chi(y)\overline{\chi(y)}=\frac{1}{\Size{G}}\sum_{y\in D}\chi(y)\overline{\chi(y)} \le  \frac{1}{\Size{G}}\sum_{y\in D}\chi(1)^2=\frac{\Size{D}}{\Size{G}}\chi(1)^2,
  \end{equation*}
  which yields $\Index{G}{D} \le \chi(1)^2$.  This completes the proof
  of item~\ref{lem:reduction_to_maximal_itemtwo_new}.  \smallskip

  We deal  now with item~\ref{lem:reduction_to_maximal_itemthree_new}.
  The previous inequality  is strict if and only  if $D\nleq Z(\chi)$,
  in which case $\Size{\chi(y)}^2 < \chi(1)^2$ for some $y\in D$, thus
  $D\le   Z(\chi)$.   As   above  it   is   easy  to   see  that,   if
  $c\in  C\setminus  Z(G)$,  the  fact  that  $c$  has  a  non-trivial
  commutator in $N$ implies $\chi(c)=0$.  This forces $C\cap D=Z(G)$.
  % Since we are assuming $C\ne Z(G)$, it follows that $1\ne [C,C]=N$
  % so that $C$ is not abelian.
  The            group            $CD$            has            order
  $\Size{C}\Size{D}/\Size{C\cap    D}=   \Size{C}\Size{D}/\Size{Z(G)}=
  \Size{G}$  by   item~\ref{lem:reduction_to_maximal_itemtwo_new},  so
  that $G=CD$.  Since $[C,D]=1$ the group $G$ is a central product and
  we have that  the map $C\times D\to G$ defined  by $(c,d)\mapsto cd$
  is an  epimorphism.  We  deduce that  $\chi$ can  be inflated  to an
  irreducible     character     of     $C\times    D$,     that     is
  $\chi(cd)=\psi(c)\eta(d)$   for   some   $\psi  \in   \Irr(C)$   and
  $\eta\in \Irr(D)$.
  % such that $\psi(c)=\chi(c)/\eta(1)=0$ whenever
  % $c\in C\setminus Z(G)$.  In particular $\psi$ is not linear and
  % $\psi(1)\ge p$.  This gives $\eta(1)=\chi(1)/\psi(1)$.
  We  claim that  $\eta$ is  linear.  Indeed  $D\le Z(\chi)$,  that is
  $\psi(1)^2\eta(1)^2=\chi(1)^2=\chi(1\cdot  d)  \overline{\chi(1\cdot
    d)}=\psi(1)^2\eta(d) \overline{\eta(d)}$ for  all $d\in D$, giving
  $Z(\eta)=D$.  Hence  $D/\ker \eta$ is  cyclic and $D'\le  \ker \eta$
  and $\eta$ is linear. In particular if $\chi(1)\ne 1$ then $\psi$ is
  not   linear   and    $C'\ne   1$.    If   $1\ne    z\in   N$   then
  $\eta(z) =  \epsilon$ and $\psi(z) =  \delta\psi(1)$, where $\delta$
  and $\epsilon$ are  roots of unity, and since $z$,  as an element of
  $C$, is identified in the central  product with $z$ as an element of
  $D$,  we  have $\delta  =  \epsilon$.   If $z\in  \ker(\eta)$,  then
  $\delta =  \epsilon = 1$  and $\chi(z)  = \psi(z) =  \chi(1)$, i.e.\
  $z\in \ker(\chi)$, a contradiction. Thus
  % $\eta(z)=\chi(z)/\psi(1)\ne \chi(1)/\psi(1)=\eta(1)$ so that
  $N\nleq       \ker        \eta$.        It        follows       that
  $N\cap D'\le  N\cap \ker  \eta=1$, as  $N$ is  of order  $p$.  Hence
  $G'=C'D'\le ND'$.

  Assume  now that  $G$  is  stem and  let  $w\in C_D(D/N)$.   Clearly
  $[w,d]\in N\cap D'=1$ for all $d\in D$.  This yields $C_D(D/N)=Z(D)$
  which                          implies                          that
  $Z(D/N)=Z(D)/N \le  Z(G)/N \le  G'/N =ND'/N\le (D/N)'$.   This shows
  that $D/N$ is a stem group which is isoclinic to $D$, since $N$ is a
  central subgroup intersecting $D'$ trivially.

  Let $g\in C\setminus Z(G)$ and let $M^*=C_G(g)$, and assume that $G$
  is stem and $\chi(1)>1$.  Suppose by way of contradiction that there
  exists $h\in  D$ such that  $b_D(h)=b$.  In particular  the equality
  $b_G(h)=b_D(h)$     holds     and     $G=DC_G(h)=M^*C_G(h)$.      By
  Lemma~\ref{2nd_lemma_abelian_breadth_1}      we       have      that
  $N\subseteq [h,G]=  [h, CD]  =[h,D] \subseteq D'$,  a contradiction.
  This   implies   $b(D)<b(G)$,   completing   thus   the   proof   of
  item~\ref{lem:reduction_to_maximal_itemthree_new}.  \medskip

  In order  to prove the  remaining items  we may assume  $C\ne Z(G)$.
  \smallskip

  If $C$  is not abelian  then there exist  $c_1, c_2\in C$  such that
  $1\neq [c_1,c_2]=z\in  N$. Let $g=c_1$, since  $c_2\notin C_G(g)$ we
  have  $\min   \Set{b(x)\mid  x\in  G\setminus  C_G(g)}\le   1$,  and
  item~\ref{lem:reduction_to_maximal_itemfour_new} is proven.

  To prove item~\ref{lem:reduction_to_maximal_itemfive_new} first note
  that if $h\notin D$ then $h$  has a non-trivial commutator in $N$ so
  that $b_{G/N}(hN)=b_G(h)-1$.  Suppose now that $b_{G/N}(hN)=b$, from
  the  argument  above  we  have   that  $h\in  D$.   Let  $K=C_G(h)$,
  $g\in  C\setminus  Z(G)$  and  $M^*=C_G(g)$.  As  in  the  proof  of
  Lemma~\ref{2nd_lemma_abelian_breadth_1},     if    $KM^*=G$     then
  $N\subseteq [G,h]$,  giving $b_{G/N}(hN)=b-1$, a  contradiction.  It
  follows   that  $K\le   D=\bigcap_{g\in   C}C_G(g)$.   This   yields
  $p^b=p^{b(h)}=\Index{G}{K}\ge \Index{G}{D}=\Index{C}{Z(G)}$.

  By  item~\ref{lem:reduction_to_maximal_itemtwo_new}   we  know  that
  $\Index{G}{D}\le \chi(1)^2\le  p^{2d}$ for every $\chi  \in \Irr(G)$
  such  that  $N\nleq  \ker(\chi)$.   Since  there  is  at  least  one
  irreducible character of $G$ whose  kernel does not contain $N$, the
  equality  $\Index{G}{D}=   p^{2d}$  would  imply  that   every  such
  characters  $\chi$ would  have  degree \(p^d>1\)  as  $N\le G'$  and
  \(\ker      \chi\)     does      not     contain      \(N\).      By
  item~\ref{lem:reduction_to_maximal_itemthree_new}   $C'\ne   1$,   a
  contradiction.
\end{proof}

\begin{remark}\label{rem:B1_not_abelian}
  Suppose that there  exist two non-commuting elements $h$  and $k$ of
  breadth $1$. Since  $C_{G}(h)$ is a maximal,  hence normal, subgroup
  of  $G$, it  contains  $[h,k]$  which then  commutes  with $h$  and,
  similarly, with $k$: clearly it  also commutes with every element in
  $L=C_{G}(h)\cap C_{G}(k)$.  As $G$ is generated by $h$, $k$ and $L$,
  the commutator $[h,k]$ is  central and thus $\Span{[h,k]}\le [G,h]$:
  since $[G,h]$ has  size $p$, equality holds and  $[G,h]$ is central,
  that  is $h\in  Z_{2}(G)$. Similarly  $\Span{[h,k]}=[G,k]$, so  that
  \(C_G(G/[G,h])\)  contains both  $h$  and  $k$ and  then  it is  not
  abelian.
\end{remark}

% \begin{remark}\label{rem:B1_not_abelian}
%   Suppose  that \(B_1(G)\) is not  an abelian subgroup of  \(G\), so
%     that  there   exist  two   elements  \(h,k\in   G\)  such   that
%     \([h,k]\ne   1\)   and  \(\br(h)=\br(k)=1\).   The   centralizer
%   \(L=C_G(\Span{h,k})\),  being the intersection of  the two maximal
%   subgroups  which are  the centralizers  of \(h\)  and \(h\),  is a
%    normal subgroup  of index  \(p^2\). In  particular \(G/L\)  is an
%                   abelian                  group                 and
%   \(\Span{[h,k]}=  [G,h]\le L\cap  \Span{h,k}\le Z(G)\).  Thus every
%   element  \(h\) of  breadth \(1\)  such that  \(h\notin Z(B_1(G))\)
%   belongs to \(Z_2(G)\). Moreover the subgroup \(C=C_G(G/[G,h])\) is
%   not abelian and  \(C'=[G,h]\). Conversely if \(C=C_G(G/[G,h])\) is
%   not abelian for some element \(h\in Z_2(G)\) of breadth \(1\) then
%   \(C\le B_1(G)\), and \(B_1(G)\) is not abelian.
% \end{remark}

\begin{proposition} \label{prop:bounds} Let $H$ be a stem $p$-group of
  breadth-degree   type  $(b,d)$,   with  $b\ge   2$,  and   of  order
  $|H|=p^k$. Suppose that $G$ is any  stem quotient of $H$ of order at
  least       $\Size{G}       \ge      p^{k-b+1}$       and       that
  \(S=\Set{\smash{g\in Z_2(G)  \mid \br(g)=1}}\)  is not  empty (\(G\)
  exists by  Lemma~\ref{lem:genbound_first_reduction}).  For  $g\in S$
  set     $N_g=[g,G]$,     $C_g=C_G(G/N_g)$,    \(M_g=C_G(g)\)     and
  $D_g=C_G(C_g)$. %One (and only one) of the following cases occurs.
  % \begin{enumerate}
  % \item \label{bound_one} \(B_1(G)\) is not abelian and for some
  %   \(g\in S\) \begin{enumerate}
  %  	\item the subgroup \(C_g\) is not abelian,
  %  	\item $1\ne \Index{G}{D_g}= \chi(1)^2 \le p^{2d}$ for some
  %     $\chi\in \Irr(G)$ such that $N_g\nleq \ker \chi$.
  %   \end{enumerate} In this case $ p^k\le p^{b+2d} \sigma(b-1,d)$.
  % \item \(B_1(G)\) is not abelian and
  %   $\Index{G}{D_g}\ne \chi(1)^2\le p^{2d}$ for all
  %   $\chi\in \Irr(G)$ such that $N_g\nleq \ker \chi$ and for all
  %   \(g\in S\) such that \(C_g\) is not abelian.  In this case
  %   $p^{k}\le p^{b+2d}\sigma(b-1,d)$.
  % \item \label{bound3}$B_1(G)$ is abelian and
  %   \[p^k \le
  %     \begin{cases}
  %       p^{b+2d-1}\sigma(b-1,d) & \text{if
  %       \(\min_{g\in S}\br(G/N_g)\le b-1\),} \\
  %       p^{3b-1+\min\{b,2d-1\}} \sigma(b-1,d) & \text{if
  %    	  \(\min_{g\in S} \br(G/N_g)=b\). }
  %     \end{cases}
  %   \]
  % \end{enumerate}
  \begin{enumerate}
  \item   \label{bound_one}  If   \(B_1(G)\)  is   not  abelian   then
    $ p^k\le p^{b+2d} \sigma(b-1,d)$.
  \item \label{bound3}If $B_1(G)$ is abelian then
    \[p^k \le
      \begin{cases}
        p^{b+2d-1}\sigma(b-1,d) & \text{if \(\min_{g\in S}\br(G/N_g)\le b-1\),} \\
        p^{3b-2}      \sigma(b-1,d)     &      \text{if
          \(\min_{g\in S} \br(G/N_g)=b\). }
      \end{cases}
    \]
  \end{enumerate}
  In particular \(p^k\le p^{3b+2d-2}\sigma(b-1,d)\) in each of the two
  cases if \(b \ge 2\) and \(d \ge 1\).

%
%
  % For any given pair $(b,d)\in \N\times \N$ the size $\sigma(b,d)$
  % of the breadth-degree type $(b,d)$ is finite and
%
%\begin{equation*}
%  \sigma(b,d)\le
%  \begin{cases}
%    p^{2b^2+b+2d-2} & \text{for $b\le 2d-1$} \\ p^{\frac{1}{2}(3b^2+b(4d-1)-4d^2+10d-6)}& \text{for $b> 2d-1$}\\
%  \end{cases}
%\end{equation*}
%%
%% $\sigma(b,d) \le p^{\frac{3b^2+\br(4d+1)-2}{2}}$ (where $b\ge 1$).
\end{proposition}

\begin{proof}
  In the rest  of the proof, for  the sake of brevity  of notation, we
  shall  use the  letters  \(N\),  \(M\), \(D\)  and  \(C\) to  denote
  respectively  \(N_g\),  \(M_g\),  \(D_g\) and  \(C_g\)  when  \(g\in
  S\). We already saw that $N$ is  a minimal normal subgroup of $G$ of
  order $p$ and  that $M\lessdot G$ is a maximal  subgroup.  Note that
  that $C\ne Z(G)$ as \(g\in C\setminus Z(G)\).

  We  now   apply  item~\ref{lem:reduction_to_maximal_itemone_new}  of
  Lemma~\ref{lem:reduction_to_maximal_new} to the group $G/N$. We have
  \begin{equation}\label{eq:missing_to_stem}
    \begin{split}
      \Index{Z(M/N)(M/N)'}{(M/N)'}& \le p^{2t} \cdot\Index{Z(G/N)(G/N)'}{(G/N)'}\\
      & = p^{2t} \cdot\Index{Z(G/N)}{\bigl(Z(G/N)\cap (G/N)'\bigr)}\\ &=p^{2t} \cdot\Index{C/N}{\bigl((C/N)\cap (G/N)'\bigr)}\\ & = p^{2t} \cdot\Index{C}{ (C\cap G')} \\
      &\le p^{2t} \cdot \Index{C}{Z(G)}.
    \end{split}
  \end{equation}
  where $t=\min\Set{b_{G/N}(xN)  \mid x\in  G\setminus M} \le  b-1$ by
  Lemma~\ref{2nd_lemma_abelian_breadth_1}.
	
  Let  us start  with  the case  when \(B_1(G)\)  is  not abelian.  By
  Remark~\ref{rem:B1_not_abelian}, there  exists \(g\in S\)  such that
  \(C\) is not abelian.

  % Indeed if \(C_g\) is not abelian then \(N_g=C_g'\) so that
  % \(C_g\le B_1(G)\), which in turn implies that \(B_1(G)\) is not
  % abelian.  Conversely if \(B_1(G)\) is not abelian and \(h,k\) are
  % two non-commuting elements of breadth \(1\) then \(G\) is the
  % central product of \(U=\Span{h,k}\) and \(V=C_G(h,k)\).  The
  % subgroup \(U\cap V\) is then central and contains \([h,k]\) as
  % \(V\) is the intersection of the two maximal subgroups \(C_G(h)\)
  % and \(C_G(k)\). Note that \(\br(h)=\br(k)=1\) implies that \(h\)
  % and \(k\) are central modulo \(N_h=\Span{[h,k]}\) so that
  % \(h\in S\) and \(C_h\) is not abelian.

  Suppose  first that  for  some \(g\in  S\) such  that  \(C\) is  not
  abelian  and some  $\chi\in \Irr(G)$  such that  $N\nleq\ker\chi$ we
  have    $1\ne    \Index{G}{D}=    \chi(1)^2    \le    p^{2d}$.    By
  item~\ref{lem:reduction_to_maximal_itemthree_new}                 of
  Lemma~\ref{lem:reduction_to_maximal_new}, the quotient $D/N$ is stem
  and                $1\ne                 N=C'$.                 Also
  $p^{k-b+1}/(p^{2d}\cdot p)\le  \Size{G}/(\Size{G/D}\cdot \Size{N}) =
  \Size{D/N} \le \sigma(b-1,d)$, which implies that
  \begin{equation*}
    p^k\le p^{b+2d} \sigma(b-1,d)
  \end{equation*}
  as claimed.

  Suppose now that for every \(g\in S\) such that \(C\) is not abelian
  and  every  $\chi\in \Irr(G)$  such  that  $N\nleq\ker\chi$ we  have
  $\Index{G}{D}\ne                \chi(1)^2\le                p^{2d}$.
  Item~\ref{lem:reduction_to_maximal_itemtwo_new}                   of
  Lemma~\ref{lem:reduction_to_maximal_new}         implies        that
  $\Index{C}{Z(G)}=\Index{G}{D}             \le             p^{2d-1}$.
  Lemma~\ref{2nd_lemma_abelian_breadth_1}  yields   $\br(M/N)<b$.   As
  \( C\nleq M\) then $\br_{G/N}(yN)=0$ for some \(y\in G\setminus M\),
  so that $t=0$.  We also have $\Size{M/N}=\Size{G}/p^2\ge p^{k-b-1}$.
  Applying \eqref{eq:missing_to_stem} with $t=0$ we find
  \begin{equation*}
    \Index{Z(M/N)(M/N)'}{(M/N)'} \le p^{2t} \cdot \Index{C}{Z(G)}\le  p^{2d-1} .
  \end{equation*}
  % where, the last two inequalities depend on the fact that $\bar G$
  % is stem, in particular we can use
  % item~\ref{lem:reduction_to_maximal_itemtwo_new} of
  % Lemma~\ref{lem:reduction_to_maximal_new}.
  By   Remark~\ref{rem:isoclinicgroupsizes},   the  group   $M/N$   is
  isoclinic to a stem group of order at least %p^{k-3b-2d-1}
  % the stem group %$T$
  % in the isoclinism class of $M/N$ has size
  \begin{equation*}
    \dfrac{\Size{M/N}}{\Size{Z(M/N)(M/N)'/(M/N)'}}\ge \dfrac{\Size{M/N}}{p^{2d-1}} \ge p^{k-b-1}/p^{2d-1}=p^{k-b-2d}.
  \end{equation*}
  Hence, as claimed, also in this case we have
  \begin{equation*}
    p^{k}\le p^{b+2d}\sigma(b-1,d).
  \end{equation*}

  Finally we have  to deal with the case when  $B_1(G)$ is abelian. If
  \(g\in   S\)   then  $C$   is   abelian   and   we  can   use   item
  \ref{lem:reduction_to_maximal_itemfive_new}                       of
  Lemma~\ref{lem:reduction_to_maximal_new}.           We          have
  $\Index{C}{Z(G)}=\Index{G}{D}        \le         p^{2d-1}$        as
  $\Index{G}{D}\ne  \chi(1)^2$ for  all  $\chi\in  \Irr(G)$ such  that
  $N\nleq  \ker  \chi$.  Suppose  first  that  $\br(G/N)\le b-1$.   By
  Remark~\ref{rem:isoclinicgroupsizes} the group $G/N$ is isoclinic to
  a stem group whose order is
  \begin{multline*}
    \frac{\Size{G/N}}{\Size     {Z(G/N)/(Z(G/N)\cap    (G/N)')}}     =
    \frac{\Size{G/N}}{\Size{C/(C\cap G')}}
    \ge \frac{\Size{G/N}}{\Size{C/ Z(G)}} \\
    \ge p^{k-b}/p^{2d-1}=p^{k-b-2d+1},
  \end{multline*}
  this gives
  \begin{equation*}
    p^k\le p^{b+2d-1}\sigma(b-1,d)
  \end{equation*}
  when $\br(G/N)\le b-1$.

  The     other    possibility     is    that     $b(G/N)=b$.    In this case there exists \(y\in G\) such that \(b=\br_G(y)=\br_{G/N}(yN)\). From Lemma~\ref{lem:genbound_first_reduction} it follows that \(y\in M\) and that \(\br(M/N)\le b-1\). As a consequence either \([M,y]\supseteq N\) or \(\br_M(y)=b-1\). Since \([G,y]\) does not contain \(N\), the only possible case is the latter one, thus \(\br(M)= b-1\) and, by item~\ref{lem:reduction_to_maximal_itemone_new} of Lemma~\ref{lem:reduction_to_maximal_new}, 
   \begin{equation*}
  	\dfrac{\Size{M}}{\Size{Z(M)M'/M'}}\ge
  	p^{k-b}/p^{2(b-1)}=p^{k-3b+2}.
  \end{equation*}
		Hence \begin{equation*}
			p^k \le p^{3b-2} \sigma(b-1,d).
		\end{equation*}
	and item~\ref{bound3} follows.
\end{proof}
\begin{remark}
  It  is easy  to see  that a  stem $p$-group  of breadth-degree  type
  $(1,d)$ is extra-special of  order $p^{2d+1}$. Indeed it has derived
  subgroup of  order $p$ containing  the center $Z(G)$.   This implies
  $Z(G)=G'=\Frat(G)$. Hence $G$ has order $p^{2n+1}$ and its character
  degrees are $1$ and $p^n$.  As a consequence $\sigma(1,d)=p^{2d+1}$.
\end{remark}
\begin{theorem}\label{thm:main}
  For any given pair $(b,d)\in  \N\times \N$ the size $\sigma(b,d)$ of
  the     breadth-degree     type     $(b,d)$    is     finite     and
  $\sigma(b,d)\le p^{\frac{b(3b+4d-1)}{2}}$.
\end{theorem}

\begin{proof}
  Proposition~\ref{prop:bounds}                                 yields
  $  \sigma(b,d)\le  p^{3b+2d-2}  \sigma(b-1,d)$   for  $b\ge  2$  and
  $d\ge 1$. By induction it follows that
  \begin{equation*}\label{eq:upper_bound}
    \sigma(b,d)\le p^{\frac{3b^2+4db-4d-b-2}{2}}\sigma(1,d)= p^{\frac{b(3b+4d-1)}{2}}.
  \end{equation*}
  since $\sigma(1,d)=p^{1+2d}$ as we already noted.
\end{proof}

% \begin{remark}\label{better_upper_bound}
%   Proposition~\ref{prop:bounds} actually shows that
%   \begin{equation*}
%    \sigma(b,d)\le
%    \begin{cases}
%      p^{b+2d+2}\sigma(b-1,d) & \text{for $2\le b \le \left\lfloor\frac{2}{3}d+1\right\rfloor$}\\
%      p^{4b-1}\sigma(b-1,d) & \text{for $\left\lfloor\frac{2}{3}d+1\right\rfloor < b \le 2d-1$} \\
%      p^{3b+2d-2}\sigma(b-1,d) & \text{for $ b > 2d-1$}
%    \end{cases}
%  \end{equation*}
%  Arguing induction we then find the more accurate upper bound
%  \begin{equation*}
%    \sigma(b,d)\le
%    \begin{cases}
%      p^{\frac{b^2+4bd+5b-4}{2}} & \text{for $2\le b \le \left\lfloor\frac{2}{3}d+1\right\rfloor$}\\
%      p^{\frac{6b^2+2d^2+9b+15d+6}{3}} & \text{for $\left\lfloor\frac{2}{3}d+1\right\rfloor < b \le 2d-1$} \\
%      p^{\frac{9b^2+12bd-8d^2-3b+72d-6}{6}} & \text{for $ b > 2d-1$}
%    \end{cases}
%  \end{equation*}
%\end{remark}

% \begin{remark}\label{rem:upper_bound}
%   The   proof   of   Theorem~\ref{thm:main}   shows   that   letting
%   $\Index{C}{Z(G)}\le      2d-1$     and      $t\le     b-1$      in
%   \eqref{eq:missing_to_stem}         it         follows         that
%   $ \sigma(b,d)\le  p^{3b+2d-2} \sigma(b-1,d)$.  This  inequality is
%   indeed satisfied  in all of  the cases  treated in the  proof.  It
%   follows that
%   \begin{equation}\label{eq:upper_bound}
%    \sigma(b,d)\le p^{\frac{3b^2+4db-4d+b-4}{2}}\sigma(1,d)\le p^{\frac{3b^2+b(4d+1)-2}{2}}
%  \end{equation}
%  is another  valid bound which,  although less accurate, has  a more
%  compact form.
%\end{remark}
%
%\begin{remark}\label{rem:unbounded_size}
\section{Examples and lower bounds}\label{sec:unbounded_size}
% In this section $p$ will denote an odd prime.
%
Consider  the  additive group  $H=\Z_p[\theta]$  of  the ring  of  the
$p$-adic integers extended by a $p$-th root of unity $\theta\ne 1$ and
let  $g$ be  the  automorphism of  order  $p$ of  $H$  induced by  the
multiplication  by $\theta$.   It is  well known  that the  semidirect
product  $M_\infty=\Span{g}\ltimes H$  is a  pro-$p$ group  of maximal
class,  i.e.\ any  normal subgroup  of finite  index of  $M_\infty$ is
either       a       maximal       subgroup      or       a       term
$\gamma_i(M_\infty)=(\theta-1)^{i-1}H$ of the  lower central series of
$M_\infty$ (see also  \cite[Example 7.4.14]{Leedham-Green2002}).  Also
the finite quotient  $M_i=M_{\infty}/\gamma_i(M_\infty)$ is of maximal
class  and has  order $\Size{M_i}=p^i$.   Since $M_i$  has an  abelian
maximal subgroup $N$, the character degrees  of $M_i$ are $1$ and $p$.
It is  also easy to  see that if $x\notin  N$ the $x^{M_i}=x  M_i'$ so
that $\br(M_i)=  \log_p(\Size{M_{i}'})=i-2$.  We  have that  $M_i$ has
breadth-degree type $(i-2,1)$ for $i\ge 3$.  Let $b\ge d$ and consider
the group  $T_{b,d}=M_{b-d+3}\times \prod_{j=1}^{d-1}E$, where  $E$ is
an  extra-special group  of  order  $p^3$.   It's  easy to  see  that
$T_{b,d}$ is stem  of order $p^{b+2d}$ and that  it has breadth-degree
type $(b,d)$.  Hence $\sigma(b,d)\ge p^{b+2d}$ for $b \ge d$.

Let $ H_{p^d}$ be the group of the unitriangular $3 \times 3$ matrices
over the Galois field $\F_{p^d}$ and $N\le Z(H_{p^d})$ be any subgroup
of order $p^{d-b}$  of its center.  The  quotient $H_{p^d}/N$ provides
an example of  a stem $p$-group of breadth-degree type  $(b,d)$ and of
order  $p^{b+2d}$. Hence  $\sigma(b,d)\ge p^{b+2d}$  also in  the case
$d > b$.

It follows that for every given $b$ and $d$ we have
\begin{equation}\label{eq:notsuitable}
  \sup_{s\ge 1}\sigma(s,d)=\sup_{t\ge 1}\sigma(b,t)=\infty.
\end{equation}

When  $p$  is  odd  another  interesting example  is  the  stem  group
$\mathcal{F}_{b+1}=F/\gamma_3(F)F^p$                             where
$F=\Span{x_1,\ldots,   x_{b+1}}$   is   the  free   group   on   $b+1$
generators. A  direct computation  shows that  $\mathcal{F}_{b+1}$ has
breadth           $b$          and           derived          subgroup
$\mathcal{F}_{b+1}'=Z(\mathcal{F}_{b+1})$  of  maximum possible  order
$p^{\frac{b(b+1)}{2}}$.                                            Let
$d=\left\lfloor \frac{b+1}{2}\right\rfloor$.   On the one  hand, since
the  square   of  the   degree  of   any  irreducible   characters  of
$\mathcal{F}_{b+1}$  divides   the  index $p^{b+1}$  of the  center
$Z(\mathcal{F}_{b+1})$      \cite[Corollary~2.30]{Isaacs1994},     the
representation exponent  of the  group $\mathcal{F}_{b+1}$ is  at most
$d$.   On the  other hand,  the  group $H_{p^{d}}  $ is  a quotient  of
$\mathcal{F}_{b+1}$,   so   that   the  representation   exponent   of
$\mathcal{F}_{b+1}$ is at least $d$. We deduce that the representation
exponent        of        $\mathcal{F}_{b+1}$        is        exactly
$d={\left\lfloor \frac{b+1}{2}\right\rfloor}$.  This  yields the lower
bound
\[\sigma\left(b,\left\lfloor \frac{b+1}{2}\right\rfloor\right)\ge
  \Size{\mathcal{F}_{b+1}}   =    p^{\frac{b^2+3b+2}{2}}\text{\   ($p$
    odd)}.\]   Since   $d'\le   d$    and   $b'\le   b$   imply   that
$\sigma(b',d')\le \sigma(b,d)$, for $b\le 2d-1$ we have that
\begin{equation*}
  p^{\frac{b^2+3b+2}{2}} \le \sigma(b,d) \le  p^{\frac{b(3b+4d-1)}{2}}\le p^{10d^2-9d+2}\text{\  ($p$ odd)},
\end{equation*}
whereas for $b\ge 2d$ we have
\begin{equation*}
  p^{d(2d+1)}\le p^{\frac{(2d-1)^2+3(2d-1)+2}{2}}\le \sigma(b,d) \le p^{\frac{b(3b+4d-1)}{2}} \le p^{\frac{5b^2+b}{2}} \text{\  ($p$ odd).}
\end{equation*}

% \end{remark}

\bibliographystyle{plain}
\bibliography{breadth_degree}

% \affiliationone{% in this example, two authors share an institution
% Norberto Gavioli\\
% Dipartimento di Ingegneria e Scienze dell'Informazione e Ma\-te\-ma\-ti\-ca\\
% Universit\`a degli Studi dell'Aquila\\
% Via Vetoio\\
% I-67100 Coppito (L'Aquila)\\
% Italy \email{norberto.gavioli@univaq.it}}
%% Important: Do not put any empty line here.
% \affiliationtwo{% in this example, two authors share an institution
% Avinoam Mann\\
% Einstein Institute of Mathematics\\
% The Hebrew University of Jerusalem\\
% Manchester House 205\\
% Givat         Ram.         Jerusalem,         9190401,        Israel
% \email{avinoam.mann@mail.huji.ac.il}}
% \affiliationthree{Valerio Monti\\
% Dipartimento di Scienza e Alta Tecnologia\\
% Universit\`a degli Studi dell'Insubria\\
% Via Valleggio, 11 \\ I-22100 Como\\
% Italy  \email{valerio.monti@uninsubria.it} }  Important: Do  not put
% any  empty  line  here.   Use \affiliationthree{}  for  any  address
% positioned  under  \affiliationone  Use \affiliationfour{}  for  any
% address positioned under \affiliationtwo

\end{document}